\documentclass[11pt]{amsart}
\usepackage{amsfonts,amssymb,amsthm,cite,amsmath,amstext}
\usepackage{color}
\usepackage[colorlinks,linkcolor=black,citecolor=black]{hyperref}
\usepackage{comment}
\usepackage[margin=2.5cm]{geometry}

\numberwithin{equation}{section}

\theoremstyle{definition}

\newtheorem{theorem}{Theorem}[section]
\newtheorem{remark}{Remark}[section]
\newtheorem{problem}{Problem}[section]
\newtheorem{conjecture}{Conjecture}[section]
\newtheorem{lemma}{Lemma}[section]

\newtheorem{definition}{Definition}[section]
\newtheorem{proposition}{Proposition}[section]
\newtheorem{example}{Example}[section]

\DeclareMathOperator{\vol}{vol}
\DeclareMathOperator{\Mov}{Mov}

\title{Movable Intersection And Bigness Criterion}
\author{Jian Xiao}
\date{}
\begin{document}
\maketitle

\begin{abstract}
In this note, we give a Morse-type bigness criterion for the difference of two pseudo-effective $(1,1)$-classes by using movable intersections. As an application, we give a Morse-type bigness criterion for the difference of two movable $(n-1,n-1)$-classes.
\end{abstract}


\section{Introduction}
Let $X$ be a smooth projective variety of dimension $n$, and let $A, B$ be two nef line bundles, then we have the fundamental inequality
$$\vol(A-B)\geq A^n -nA^{n-1}\cdot B,$$
which is first discovered as a consequence of Demailly's holomorphic Morse inequalities (see \cite{Dem85}, \cite{Siu93}, \cite{Tra95}). Thus the above inequality is usually called algebraic Morse inequality for line bundles. Recall that the volume of a holomorphic line bundle $L$ is defined as
$$
\vol(L):=\limsup_{k\rightarrow +\infty} \frac{n!}{k^n}h^0 (X, \mathcal{O}(kL)).
$$
And $L$ is called a big line bundle if $\vol(L)>0$. In particular, the Morse-type inequality for $A,B$ implies that $A-B$ must be a big line bundle if $A^n -nA^{n-1}\cdot B>0$. This provides a very effective way to construct holomorphic sections; see \cite{DMR10, Dem11} for related applications.

Assume that $L$ is a holomorphic line bundle over a compact K\"ahler manifold $X$, then it is proved by \cite[Theorem 1.2]{Bou02a} that the volume of $L$ can be characterized as the maximum of the Monge-Amp\`{e}re mass of the positive curvature currents contained in the class $c_1(L)$.
This naturally extends the volume function $\vol(\cdot)$ to transcendental $(1,1)$-classes over compact complex manifold (see \cite[Definition 1.3]{Bou02a} or \cite[Definition 3.2]{BDPP13}).

Recall that Demailly's conjecture on (weak) transcendental holomorphic Morse inequality over compact K\"ahler manifolds is the following statement.

\begin{conjecture} 
(see \cite[Conjecture 10.1]{BDPP13})\footnote{For projective manifolds, this conjecture has been confirmed by \cite{nystrom2016duality}.}
Let $X$ be a compact K\"ahler manifold of dimension $n$, and let $\alpha, \beta\in H^{1,1}(X,\mathbb{R})$ be two nef classes. Then we have
$$\vol(\alpha-\beta)\geq \alpha^n -n \alpha^{n-1}\cdot \beta.$$
In particular, if $\alpha^n -n \alpha^{n-1}\cdot \beta>0$ then there exists a K\"ahler current in the class $\alpha-\beta$.
\end{conjecture}

Based on the method of \cite{chiose2016kahler}, in our previous work \cite{xiao13weakmorse}, it was proved that if $\alpha^n -4n \alpha^{n-1}\cdot \beta>0$ then there exists a K\"ahler current in the class $\alpha-\beta$. Recently, by keeping the same method as in \cite{chiose2016kahler, xiao13weakmorse} and with the new technique introduced by \cite{popovici2016sufficientbig}, D.~Popovici proved that the constant $4n$ can be improved to be the natural and optimal constant $n$. Thus we have a Morse-type criterion for the difference of two transcendental nef classes -- indeed, our results in this note depend mainly on this important improvement.

It is natural to ask whether the above Morse-type bigness criterion ``$\alpha^n -n \alpha^{n-1}\cdot \beta>0 \Rightarrow \vol(\alpha-\beta)>0$" for nef classes can be generalized to pseudo-effective $(1,1)$-classes. Towards this generalization, we need the movable intersection products (denoted by $\langle-\rangle$) of pseudo-effective $(1,1)$-classes developed in \cite{Bou02b, BDPP13}. Then our problem can be stated as following:

\begin{problem}\label{problem}
Let $X$ be a compact K\"ahler manifold of dimension $n$, and let $\alpha, \beta\in H^{1,1}(X,\mathbb{R})$ be two pseudo-effective classes. Does $\vol(\alpha) -n \langle \alpha^{n-1} \rangle \cdot \beta>0$ imply that there exists a K\"ahler current in the class $\alpha-\beta$?
\end{problem}

Unfortunately, a very simple example due to \cite{Tra95}
implies that the above generalization does not always hold.

\begin{example}
(see \cite[Example 3.8]{Tra95})
\label{eg trapani}
Let $\pi: X\rightarrow \mathbb{P}^2$ be the blow-up of $\mathbb{P}^2$ along a point $p$. Let $R=\pi^* H$, where $H$ is the hyperplane line bundle on $\mathbb{P}^2$. Let $E=\pi^{-1}(p)$ be the exceptional divisor. Then for every positive integral $k$, the space of global holomorphic sections of $k(R-2E)$ is the space of homogeneous polynomials in three variables of degree at most $k$ and vanishes of order $2k$ at $p$; hence $k(R-2E)$ does not have any global holomorphic sections. The space $H^0 (X, \mathcal{O}(k(R-2E)))=\{0\}$ implies that $R-2E$ can not be big. However, we have $R^2-R\cdot 2E >0$ as $R^2=1$ and $R\cdot E=0$.
\end{example}

As the first result of this note, we show that the answer to Problem \ref{problem} is YES if $\beta$ is movable. Here $\beta$ being movable means that the negative part of $\beta$ vanishes in its divisorial Zariski decomposition (see \cite{Bou04}).
In particular, if $\beta=c_1(L)$ for some pseudo-effective line bundle, then $\beta$ being movable is equivalent to that the base locus of $mL+A$ is of codimension at least two for a fixed ample line bundle $A$ and for large $m$.

\begin{theorem}
\label{thm mov-big}
Let $X$ be a compact K\"ahler manifold of dimension $n$, and let $\alpha, \beta\in H^{1,1}(X, \mathbb{R})$ be two pseudo-effective classes with $\beta$ movable. Then $\vol(\alpha)-n \langle\alpha^{n-1} \rangle \cdot \beta>0$ implies that there exists a K\"ahler current in the class $\alpha-\beta$.
\end{theorem}

\begin{remark}
In the case when $\beta=0$, Theorem \ref{thm mov-big} is just \cite[Theorem 4.7]{Bou02a}, and when $\alpha$ is also nef, it is \cite[Theorem 0.5]{DP04}.
\end{remark}

An ancillary goal of the note is to explain the fact that Demailly's conjecture on weak transcendental holomorphic Morse inequality over compact K\"ahler manifolds is equivalent to the $\mathcal{C}^1$ differentiability of the volume function for transcendental $(1,1)$-classes. Though not stated explicitly, this fact is essentially contained in \cite[Section 3]{BFJ09} and the key ingredients are also implicitly contained in \cite{BDPP13}.

\begin{proposition}
\label{prop equiv intro}
Let $X$ be a compact K\"ahler manifold of dimension $n$. Then the following statements are equivalent:
\begin{enumerate}
  \item Let $\alpha,\beta\in H^{1,1}(X, \mathbb{R})$ be two nef classes, then we have $$\vol(\alpha-\beta)\geq \alpha^n-n\alpha^{n-1}\cdot \beta.$$
  \item Let $\alpha,\beta\in H^{1,1}(X, \mathbb{R})$ be two pseudo-effective classes with $\beta$ movable, then $$\vol(\alpha-\beta)\geq \langle\alpha^n\rangle -n\langle\alpha^{n-1}\rangle \cdot \beta.$$
  \item Let $\alpha, \gamma\in H^{1,1}(X, \mathbb{R})$ be two $(1,1)$-classes with $\alpha$ big, then we have
      $$\left. \frac{d}{dt} \right|_{t = 0} \vol(\alpha+t\gamma)
      =n\langle \alpha^{n-1}\rangle\cdot \gamma.$$
\end{enumerate}
\end{proposition}

As an application of Proposition \ref{prop equiv intro} and the $\mathcal{C}^1$ differentiability of the volume function for line bundles (see \cite[Theorem A]{BFJ09}), the algebraic Morse inequality can be generalized as following. It generalizes the previous result \cite[Corollary 3.2]{Tra11}.

\begin{theorem}
\label{thm line mov-morse}
Let $X$ be a smooth projective variety of dimension $n$, and let $\alpha, \beta$ be the first Chern classes of two pseudo-effective line bundles with $\beta$ movable. Then we have $$\vol(\alpha-\beta)\geq \vol(\alpha)-n \langle\alpha^{n-1} \rangle \cdot \beta.$$
\end{theorem}

\begin{remark}
In particular, if $\alpha$ is nef and $\beta$ is movable then we have $\vol(\alpha-\beta)\geq \alpha^n -n \alpha^{n-1}  \cdot \beta$
which is just \cite[Corollary 3.2]{Tra11}.
\end{remark}

Finally, as an application of Theorem \ref{thm mov-big}, we give a Morse-type bigness criterion for the difference of two movable $(n-1,n-1)$-classes.

\begin{theorem}
\label{thm morse n-1}
Let $X$ be a compact K\"ahler manifold of dimension $n$, and let $\alpha, \beta\in H^{1,1}(X, \mathbb{R})$ be two pseudo-effective classes. Then $\vol(\alpha)
-n \alpha\cdot \langle\beta^{n-1} \rangle >0$ implies that there exists a strictly positive $(n-1,n-1)$-current in the class $\langle \alpha^{n-1} \rangle-\langle\beta^{n-1} \rangle$.
\end{theorem}

\section{Technical preliminaries}

\subsection{Resolution of singularities of positive currents}

Let $X$ be a compact complex manifold, and let $T$ be a $d$-closed almost positive $(1,1)$-current on $X$, that is, there exists a smooth $(1,1)$-form $\gamma$ such that $T\geq \gamma$. If $\gamma=0$, then $T$ is called a positive $(1,1)$-current and the class $\{T\}$ is called pseudo-effective; And if $\gamma$ is a hermitian metric, then $T$ is called a K\"ahler current and the class $\{T\}$ is called big; see \cite{Dem_AGbook} for the basic theory of positive currents.

Demailly's regularization theorem (see \cite{Dem92}) implies that we can always approximate the almost positive $(1,1)$-current $T$ by a family of almost positive closed $(1,1)$-currents $T_k$ with analytic singularities such that $T_k \geq \gamma- \varepsilon_k \omega$, where $\varepsilon_k \downarrow 0$ is a sequence of positive constants and $\omega$ is a fixed hermitian metric. In particular, when $T$ is a K\"ahler current, it can be approximated by a family of K\"ahler currents with analytic singularities.

When $T$ has analytic singularities along an analytic subvariety $V(\mathcal{I})$ where $\mathcal{I}\subset \mathcal{O}_X$ is a coherent ideal sheaf, by blowing up along $V(\mathcal{I})$ and then resolving the singularities, we get a modification $\mu: \widetilde{X}\rightarrow X$ such that $\mu^* T= \widetilde \theta + [D]$ where $\widetilde \theta$ is an almost positive smooth $(1,1)$-form with $\widetilde \theta \geq \mu^* \gamma$ and $D$ is an effective $\mathbb{R}$-divisor; see e.g. \cite[Theorem 3.1]{BDPP13}. In particular, if $T$ is positive, then $\widetilde \theta$ is a smooth positive $(1,1)$-form.  We call such a modification the log-resolution of singularities of $T$.

For almost positive $(1,1)$-current $T$, we can always decompose $T$ with respect to the Lebesgue measure; see e.g. \cite[Section 2.3]{Bou02a}. We write $T=T_{ac}+T_{sg}$ where $T_{ac}$ is the absolutely continuous part and $T_{sg}$ is the singular part. The absolutely part $T_{ac}$ can be seen as a form with $L_{loc}^1$ coefficients, and the wedge product $T_{ac}^k (x)$ makes sense for almost every point $x$. We always have $T_{ac}\geq \gamma$ since $\gamma$ is smooth. If $T$ has analytic singularities along $V$, then $T_{ac}=\textbf{1}_{X\backslash V}T$. However, in general $T_{ac}$ is not closed even if $T$ is closed (see \cite{Bou02b}). We have the following proposition.

\begin{proposition}
\label{prop log-resolution}
Let $T_1,..., T_k$ be $k$ almost positive closed $(1,1)$-currents with analytic singularities on $X$ and let $\psi$ be a smooth $(n-k,n-k)$-form. Let $\mu: \widetilde{X}\rightarrow X$ be a simultaneous log-resolution with $\mu^* T_i =\widetilde{\theta}_i +[D_i]$. Then
$$\int_X T_{1,ac}\wedge...\wedge T_{k,ac}\wedge \psi=\int_{\widetilde{X}}\widetilde{\theta}_1\wedge ...\wedge\widetilde{\theta}_k \wedge \mu^*\psi.$$
\begin{proof}
This is obvious since $\mu$ is an isomorphism outside 
a proper analytic subvariety and $T_{1,ac}\wedge...\wedge T_{k,ac}$ puts no mass on such subset and $\widetilde{\theta}_i$ is smooth on $\widetilde{X}$.
\end{proof}
\end{proposition}

\subsection{Movable cohomology classes}

We first briefly recall the definition of divisorial Zariski decomposition and the definition of movable $(1,1)$-class on compact complex manifold; see \cite{Bou04}, see also \cite{Nak04} for the algebraic approach.

Let $X$ be a compact complex manifold of dimension $n$ and let $\alpha $ be a pseudo-effective $(1,1)$-class over $X$, then one can always  associate an effective divisor $N(\alpha):=\sum \nu(\alpha, D)D$ to $\alpha$ where the sum ranges among all prime divisors on $X$. The class $\{N(\alpha)\}$ is called the negative part of $\alpha$. And $Z(\alpha)=\alpha-\{N(\alpha)\}$ is called the positive part of $\alpha$. The decomposition $\alpha=Z(\alpha)+\{N(\alpha)\}$ then is the divisorial Zariski decomposition of $\alpha$.

\begin{definition}
Let $X$ be a compact complex manifold of dimension $n$, and let $\alpha$ be a pseudo-effective $(1,1)$-class. Then $\alpha$ is called movable if $\alpha=Z(\alpha)$.
\end{definition}

\begin{proposition}
\label{prop modified kahler}
(see \cite[Proposition 2.3]{Bou04})
Let $\alpha$ be a movable $(1,1)$-class and let $\omega$ be a K\"ahler class, then for any $\delta>0$ there exist a modification $\mu: Y\rightarrow X$ and a K\"ahler class $\widetilde{\omega}$ over $Y$ such that $\alpha+\delta \omega=\mu_* \widetilde{\omega}$.
\end{proposition}

\begin{remark}
In \cite{Bou04}, $\alpha$ is called modified nef if $\alpha=Z(\alpha)$ (see \cite[Definition 2.2 and Proposition 3.8]{Bou04}). Here we call it movable in order to keep the same notation as the algebraic geometry situation. Let $L$ be a  line bundle over a smooth projective variety and let $\alpha=c_1(L)$, then $\alpha$ is modified nef if and only if $L$ is movable.
\end{remark}

\subsection{Movable intersections}
We take the opportunity to point out the well known fact that the several definitions of movable intersections of pseudo-effective $(1,1)$-classes over compact K\"ahler manifold coincide; see \cite{Bou02b, BDPP13, BEGZ10} for the analytic constructions over compact K\"ahler manifold. And it also coincides with the algebraic construction of \cite{BFJ09} on smooth projective variety for specified degrees. We remark that it is helpful to know the definition of movable intersections of pseudo-effective $(1,1)$-classes can be interpreted in several equivalent ways.

Let $\alpha_1,...,\alpha_k\in H^{1,1}(X, \mathbb{R})$ be pseudo-effective classes on a compact K\"ahler manifold of dimension $n$. By the common basic property of these definitions of movable intersection products, we only need to verify the respectively defined positive $(k,k)$ cohomology classes $\langle\alpha_1\cdot...\cdot \alpha_k\rangle$ coincide when all the classes are big. Firstly, by the definition of Riemann-Zariski space, it is clear from \cite[Theorem 3.5]{BDPP13} and
\cite[Definition 2.5]{BFJ09} that the two definitions of movable intersection products are the same when $X$ is a smooth projective variety defined over $\mathbb{C}$ and all the classes $\alpha_i$ are in the N\'{e}ron-Severi space and $k=1, n-1, n$. Next, by testing on $\partial\bar \partial$-closed smooth positive $(n-k, n-k)$-forms, \cite[Theorem 3.5]{BDPP13}, \cite[Definition 1.17, Proposition 1.18 and Proposition 1.20]{BEGZ10} and \cite[Definition 3.2.1 and Lemma 3.2.5]{Bou02b} imply that these three definitions give the same positive cohomology class in $H^{k,k}(X, \mathbb{R})$; see \cite[Proposition 1.10]{principato2013mobile} for the detailed proof.

Inspired by \cite[Definition 1.3, Theorem 1.5 and Conjecture 2.3]{BDPP13}, the definition of movable $(n-1,n-1)$-classes in the K\"ahler setting can be formulated as following.

\begin{definition}
Let $X$ be a compact K\"ahler manifold of dimension $n$, and let $\gamma\in H^{n-1, n-1}(X, \mathbb{R})$. Then $\gamma$ is called a movable $(n-1,n-1)$-class if it is in the closure of the convex cone generated by cohomology classes of the form $\langle \alpha_1\cdot...\cdot \alpha_{n-1}\rangle$ with every $\alpha_i$ pseudo-effective.
\end{definition}

\begin{remark}
When $X$ is a smooth projective variety of dimension $n$, \cite[Theorem 1.5]{BDPP13} implies that the rational movable $(n-1,n-1)$-classes are in the cone of movable curves.
\end{remark}


\section{Proof of the main results}

Now let us begin to prove our main results. We first give a Morse-type bigness criterion for the difference of two pseudo-effective $(1,1)$-classes by using movable intersections. To this end, we need some properties of movable intersections.

\begin{lemma}
\label{lem nef same products}
Let $X$ be a compact K\"ahler manifold of dimension $n$, and let $\alpha_1,...,\alpha_{n-1}, \beta\in H^{1,1}(X,\mathbb{R})$ be pseudo-effective classes with $\beta$ nef. Then we have
$$\langle \alpha_1 \cdot...\cdot\alpha_{n-1}\cdot \beta \rangle=\langle \alpha_1 \cdot...\cdot\alpha_{n-1}\rangle \cdot \beta.$$
\end{lemma}

\begin{proof}
By the continuity of positive products, by taking limits, we can assume that $\alpha_1,...,\alpha_{n-1}$ are big and $\beta$ is K\"ahler.

First, note that we always have $\langle \alpha_1 \cdot...\cdot\alpha_{n-1}\cdot \beta \rangle \leq \langle \alpha_1 \cdot...\cdot\alpha_{n-1}\rangle \cdot \beta$ when $\beta$ is only assumed to be big (or pseudo-effective). By \cite[Theorem 3.5]{BDPP13}, there exists a sequence of simultaneous log-resolutions $\mu_m: X_m \rightarrow X$ with $\mu_m ^* \alpha_i = \omega_{i, m}+ [D_{i,m}]$ and
$\mu_m ^* \beta= \gamma_m +[E_m]$ such that
\begin{align*}
\langle \alpha_1 \cdot...\cdot\alpha_{n-1}\cdot \beta \rangle
= \limsup_{m \rightarrow \infty}\ (\omega_{1, m}\cdot...\cdot\omega_{n-1, m} \cdot \gamma_m),
\end{align*}
and
\begin{align*}
\langle \alpha_1 \cdot...\cdot\alpha_{n-1}\rangle\cdot \beta
= \limsup_{m \rightarrow \infty}\ (\omega_{1, m}\cdot...\cdot\omega_{n-1, m}) \cdot \mu_m ^* \beta.
\end{align*}

When $\beta$ is K\"ahler, the class $\gamma_m$ is just $\mu_m ^* \beta$.
In conclusion, if $\beta$ is nef then we have the desired equality
$$\langle \alpha_1 \cdot...\cdot\alpha_{n-1}\cdot \beta \rangle=\langle \alpha_1 \cdot...\cdot\alpha_{n-1}\rangle \cdot \beta.$$
\end{proof}

\begin{lemma}
\label{lem mov ineq}
Let $X$ be a compact K\"ahler manifold of dimension $n$. Let $\alpha_1, ..., \alpha_{n-1}$ be pseudo-effective $(1,1)$-classes, and let $\pi: Y\rightarrow X$ be a modification. Then for any K\"ahler class $\widehat{\omega}$ on $Y$ we have
$$\langle\alpha_1 \cdot ...\cdot \alpha_{n-1}\rangle \cdot \pi_*  \widehat{\omega} \geq \langle\pi^*\alpha_1 \cdot ...\cdot \pi^*\alpha_{n-1}\rangle \cdot \widehat{\omega}.$$
\end{lemma}

\begin{proof}
First, we have $\langle\alpha_1 \cdot ...\cdot \alpha_{n-1}\rangle \cdot \pi_*  \widehat{\omega} \geq \langle\alpha_1 \cdot ...\cdot \alpha_{n-1} \cdot \pi_*  \widehat{\omega} \rangle$ as noted in the proof of Lemma \ref{lem nef same products}.
On the other hand, we have $\langle\alpha_1 \cdot ...\cdot \alpha_{n-1} \cdot \pi_*  \widehat{\omega} \rangle = \langle\pi^*\alpha_1 \cdot ...\cdot \pi^*\alpha_{n-1} \cdot \pi^* (\pi_*  \widehat{\omega}) \rangle$. Since $\pi^* (\pi_*  \widehat{\omega})-\widehat{\omega}$ is pseudo-effective, we get that
\begin{align*}
\langle\alpha_1 \cdot ...\cdot \alpha_{n-1}\rangle \cdot \pi_*  \widehat{\omega} &\geq \langle\pi^*\alpha_1 \cdot ...\cdot \pi^*\alpha_{n-1} \cdot \widehat{\omega} \rangle\\
&=\langle\pi^*\alpha_1 \cdot ...\cdot \pi^*\alpha_{n-1}\rangle  \cdot \widehat{\omega}.
\end{align*}

\end{proof}

\subsection{Theorem \ref{thm mov-big}}

\begin{proof}[Proof of Theorem \ref{thm mov-big}]
Fix a K\"ahler metric $\omega$ on $X$, and denote the K\"ahler class by the same symbol. By continuity and the definition of movable intersections, we have
$$
\lim_{\delta \rightarrow 0} \langle(\alpha+\delta\omega)^n \rangle-n\langle(\alpha+\delta\omega)^{n-1}\rangle
\cdot(\beta+\delta\omega)=\langle\alpha^n \rangle-n\langle\alpha^{n-1}\rangle\cdot\beta.$$
So $\langle(\alpha+\delta\omega)^n \rangle-n\langle(\alpha+\delta\omega)^{n-1}\rangle
\cdot(\beta+\delta\omega)>0$ for small $\delta>0$. Note also that $\alpha-\beta=(\alpha+\delta\omega)-(\beta+\delta\omega)$. Thus to prove the bigness of the class $\alpha-\beta$, we can assume $\alpha$ is big, and assume $\beta=\mu_* \widetilde{\omega}$ for some modification $\mu:Y\rightarrow X$ and some K\"ahler class $\widetilde{\omega}$ on $Y$ at the beginning.
By Lemma \ref{lem nef same products}
and Lemma \ref{lem mov ineq}, our assumption then implies that
$$
\langle(\mu^*\alpha)^n\rangle-n \langle(\mu^*\alpha)^{n-1}\cdot \widetilde{\omega}\rangle>0.
$$

We claim that this implies there exists a K\"ahler current in the class $\mu^*\alpha-\widetilde{\omega}$, which then implies the bigness of the class $\alpha-\beta
=\mu_*(\mu^*\alpha-\widetilde{\omega})$.

Now it is reduced to prove the case when $\alpha$ is big and $\beta$ is K\"ahler. Let $\omega$ be a K\"ahler metric in the class $\beta$. The definition of movable intersections implies there exists some K\"ahler current $T\in \alpha$ with analytic singularities along some subvariety $V$ such that
\begin{align*}
\int_{X\setminus V} T^n -n\int_{X\setminus V} T^{n-1}\wedge \omega >0.
\end{align*}
Let $\pi: Z\rightarrow X$ be the log-resolution of the current $T$ with $\pi^* T= \theta + [D]$ such that $\theta$ is a smooth positive $(1,1)$-form on $Z$. By Proposition \ref{prop log-resolution} we have
\begin{align*}
\int_Z \theta^n -n\int_Z \theta^{n-1}\wedge \pi^*\omega >0.
\end{align*}

The result of \cite{popovici2016sufficientbig} then implies that there exists a K\"ahler current in the class $\{\theta-\pi^*\omega\}$. As $\pi^* \alpha= \{\theta + [D]\}$, this proves the bigness of the class $\alpha-\beta$.

Thus we finish the proof that there exists a K\"ahler current in the general case when $\alpha$ is pseudo-effective and $\beta$ is movable.
\end{proof}

\begin{remark}
Indeed, by the above argument we have the following bigness criterion: for any pseudo-effective $(1,1)$ classes $\alpha, \beta$,
$$\langle \alpha^{n}\rangle- n \langle \alpha^{n-1} \cdot \beta \rangle >0 \Rightarrow \vol(Z(\alpha)-Z(\beta))>0.$$

Since the positive products $\langle \alpha^{n}\rangle$ and $\langle\alpha^{n-1} \cdot \beta\rangle$ depend only on the positive parts of $\alpha, \beta$ (see e.g. \cite[Proposition 3.2.10]{Bou02b}), we can assume that $\alpha, \beta$ are movable at the beginning. And by taking limits, we can also assume $\beta = \pi_* \widehat{\omega}$ for some modification $\pi$. Then we have
\begin{align*}
  \langle (\pi^* \alpha)^{n}\rangle - n\langle (\pi^* \alpha)^{n-1}\rangle \cdot \widehat{\omega} &= \langle (\pi^* \alpha)^{n}\rangle - n\langle (\pi^* \alpha)^{n-1}  \cdot \widehat{\omega}\rangle\\
  & \geq \langle (\pi^* \alpha)^{n}\rangle -n\langle (\pi^* \alpha)^{n-1}  \cdot \pi^* \beta\rangle\\
  & >0,
\end{align*}
which implies that $\alpha -\beta = \pi_* (\pi^* \alpha - \widehat{\omega})$ is big.

In the case of Example \ref{eg trapani}, since $R$ is nef and $E$ is exceptional, we have
$\langle R^2\rangle-2 \langle R\cdot 2E\rangle=R^2>0 $, we then get the bigness of $Z(R)-Z(2E)=R$.
\end{remark}

\subsection{Proposition \ref{prop equiv intro}}
Towards the transcendental version of Theorem \ref{thm morse line bdl}, we give the proof of Proposition \ref{prop equiv intro} which is essentially already known in \cite{BFJ09}.

\begin{proof} [Proof of Proposition \ref{prop equiv intro}]
It is obvious that (2)$\Rightarrow$(1). We will show that (1)$\Rightarrow$(2) and (1)$\Leftrightarrow$(3), which then proves the equivalence of the statements.

Firstly, we prove (1)$\Rightarrow$(2). To prove (2), we only need to consider the case when $\langle\alpha^n\rangle -n\langle\alpha^{n-1}\rangle \cdot \beta>0$ and $\beta$ is big and movable. By Theorem \ref{thm mov-big} we know that $\alpha-\beta$ is big.
First assume that $\beta$ is K\"ahler. By taking a sequence of suitable Fujita approximations $\mu_m$ satisfying $\mu_m ^* \alpha = \omega_m + [E_m]$ and
\begin{equation*}
  \langle \alpha^n \rangle = \limsup_{m \rightarrow \infty} \omega_m ^n, \ \ \langle \alpha^{n-1} \rangle = \limsup_{m \rightarrow \infty} (\mu_m)_* (\omega_m ^{n-1}).
\end{equation*}
Then (1) implies that $\vol(\omega_m - \mu_m ^* \beta) \geq \omega_m ^n - n \omega_m ^{n-1}\cdot \mu_m ^* \beta$. Thus (1) implies that
\begin{equation*}
  \vol(\alpha -\beta) \geq \vol(\alpha) - n \langle \alpha ^{n-1}\rangle \cdot \beta
\end{equation*}
holds in the case when $\beta$ is K\"ahler. In the general case, we can assume that $\beta = \pi_* \widehat{\omega}$. Then we have
\begin{align*}
  \vol(\alpha - \beta) \geq \vol(\pi^* \alpha - \widehat{\omega}) &\geq \vol(\pi^* \alpha) - n\langle (\pi^* \alpha) ^{n-1}\rangle \cdot \widehat{\omega}\\
  & \geq \vol(\pi^* \alpha) - n\langle  \alpha ^{n-1}\rangle \cdot \pi_*\widehat{\omega}\ \ \textrm{by Lemma \ref{lem mov ineq}}.
\end{align*}

Next, note that the implication (3)$\Rightarrow$(1) is obvious, and the implication (1)$\Rightarrow$(3) is just \cite[Section 3.2]{BFJ09}.
For reader's convenience, we briefly recall and repeat the arguments of \cite[Section 3.2]{BFJ09}. By \cite[Corollary 3.4]{BFJ09} (or the proof of \cite[Theorem 4.1]{BDPP13}), (1) implies
$$\vol(\beta+ t\gamma)\geq
\beta^n +tn\beta^{n-1}\cdot \gamma -Ct^2
$$
for an arbitrary nef class $\beta$, an arbitrary $(1,1)$-class $\gamma$ and $t\in [0,1]$. Here the constant $C$ depends only on the class $\beta,\gamma$; more precisely, the constant $C$ depends on the volume of a big and nef class $\omega$ such that $\omega-\beta$ is pseudo-effective and $\omega \pm \gamma$ is nef.

Now take a log-resolution $\mu^*\alpha = \beta + [E]$, then we have
\begin{align*}
\vol(\alpha+t\gamma)&\geq \vol(\beta+ t \mu^* \gamma)\\
&\geq \beta^n +tn\beta^{n-1}\cdot \mu^*\gamma -Ct^2\\
&= \beta^n +tn\mu_*(\beta^{n-1})\cdot \gamma -Ct^2.
\end{align*}
Note that the constant $C$ does not depend on the resolution $\mu$, since $\mu^*\omega-\beta$ is pseudo-effective and $\mu^*\omega \pm\mu^*\gamma$ is nef if $\omega$ has similar property with respect to $\alpha, \gamma$. And we have $\vol(\mu^* \omega)=\vol(\omega)$.
By taking limits of some sequence of log-resolutions, we get
\begin{align*}
\vol(\alpha+t\gamma)
\geq \vol(\alpha) +tn\langle\alpha^{n-1}\rangle\cdot \gamma -Ct^2.
\end{align*}
Replace $\gamma$ by $-\gamma$, we then get
\begin{align*}
\vol(\alpha)
\geq \vol(\alpha+t\gamma) -tn\langle(\alpha+t\gamma)^{n-1}\rangle\cdot \gamma -Ct^2.
\end{align*}
Since $\alpha$ is big, by the continuity of movable intersections (see e.g. \cite[Theorem 3.5]{BDPP13}) we have
\begin{align*}
\lim_{t\rightarrow 0} \langle(\alpha+t\gamma)^{n-1}\rangle =\langle \alpha^{n-1} \rangle.
\end{align*}
Then (3) follows easily from the above inequalities.
\end{proof}

\begin{remark}
It is shown in \cite{Den15} that the $\mathcal{C}^1$ differentiability of the volume function for transcendental $(1,1)$-classes holds on compact K\"ahler surfaces. And it is used to construct the Okounkov bodies of transcendental $(1,1)$-classes over compact K\"ahler surfaces.
\end{remark}

\subsection{Theorem \ref{thm line mov-morse}}

The algebraic Morse inequality tells us that, if $L$ and $F$ are two nef line bundles, then $$\vol(L-F)\geq L^n -nL^{n-1}\cdot F.$$ Recently, \cite{Tra11} generalizes this result to the case when $F$ is only movable. Assume that $L$ is nef and $F$ is pseudo-effective, and let $F=Z(F)+N(F)$ be the divisorial Zariski decomposition of $F$. Then
\cite[Corollary 3.2]{Tra11} shows that
\begin{equation*}
\vol(L-Z(F))\geq L^n-nL^{n-1}\cdot Z(F).
\end{equation*}
Moreover, if we write the negative part $N(F)=\sum_{j}\nu_j D_j$ where $\nu_j>0$ and let $u$ be a nef class on $X$ such that $c_1 (\mathcal{O}_{TX}(1)) +\pi^* u$ is a nef class on $\mathbb{P}(T^* X)$.
Then \cite[Theorem 3.3]{Tra11} also gives a lower bound for $\vol(L-F)$:
\begin{equation*}
\vol(L-F)\geq L^n-nL^{n-1}\cdot Z(F)-n\sum_j (L+\nu_j u)^{n-1}\cdot \nu_j D_j.
\end{equation*}

Our next result shows that $L$ can be any pseudo-effective line bundle, where we restate Theorem \ref{thm line mov-morse}.

\begin{theorem}
\label{thm morse line bdl}
Let $X$ be a smooth projective variety of dimension $n$, and let $L, M$ be two pseudo-effective line bundles with $M$ movable. Then we have $$\vol(L-M)\geq \vol(L)-n \langle L^{n-1} \rangle \cdot M.$$
\end{theorem}

\begin{proof}
This follows directly from Proposition \ref{prop equiv intro}.
\end{proof}

\begin{remark}
When $M$ is nef and $L$ is pseudo-effective, Theorem \ref{thm morse line bdl} can be also proved by using the singular Morse inequalities for line bundles (see \cite{Bon93}). Without loss of generality, we can assume $L$ is big and $M$ is ample. Let $\omega\in c_1(M)$ be a K\"ahler metric. For any K\"ahler current $T \in c_1(L)$ with analytic singularities, $T - \omega$ is an almost positive curvature current of $L-M$ with analytic singularities.
With the elementary pointwise inequality $$\mathbf{1}_{X(\alpha-\beta, \leq 1)}(\alpha-\beta)^n \geq \alpha^n - n \alpha^{n-1}\wedge \beta$$ for positive $(1,1)$-forms, Theorem \ref{thm morse line bdl} then follows easily from \cite{Bon93}.
\end{remark}

\subsection{Theorem \ref{thm morse n-1}}
Finally, inspired by the method in our previous work \cite{xiao13weakmorse} (see also \cite{chiose2016kahler}), we show that Theorem \ref{thm mov-big} gives a Morse-type bigness criterion of the difference of two movable $(n-1, n-1)$-classes, thus giving the proof of Theorem \ref{thm morse n-1}.

\begin{proof}[Proof of Theorem \ref{thm morse n-1}]
Denote the K\"ahler cone of $X$ by $\mathcal{K}$, and denote the cone generated by cohomology classes represented by positive $(n-1, n-1)$-currents by $\mathcal{N}$. Then by the numerical characterization of K\"ahler cone of \cite{DP04} (see also \cite[Theorem 2.1]{BDPP13}) we have the cone duality relation
$$
\overline{\mathcal{K}}^\vee = \mathcal{N}.
$$

Without loss of generality, we can assume that $\alpha, \beta$ are big. Then the existence of a strictly positive
$(n-1, n-1)$-current in the class
$\langle \alpha^{n-1}\rangle-\langle \beta^{n-1}\rangle$ is equivalent to the existence of some positive constant $\delta>0$ such that
\begin{align*}
\langle\alpha^{n-1}\rangle-\langle \beta^{n-1}\rangle
\succeq \delta \langle \beta^{n-1}\rangle,
\end{align*}
or equivalently,
\begin{align*}
\langle\alpha^{n-1}\rangle
\succeq (1+\delta) \langle \beta^{n-1}\rangle.
\end{align*}
Here we denote $\gamma\succeq \eta$ if $\gamma-\eta$ contains a positive current.

In the following, we will argue by contradiction. By the cone duality relation $\overline{\mathcal{K}}^\vee = \mathcal{N}$, the class
$\langle \alpha^{n-1}\rangle-\langle \beta^{n-1}\rangle$ does not contain any strictly positive
$(n-1, n-1)$-current is then equivalent to the statement: for any $\epsilon>0$ there exists some non-zero class $N_\epsilon \in \overline{\mathcal{K}}$ such that
\begin{align*}
\langle \alpha^{n-1}\rangle\cdot N_\epsilon \leq
(1+\epsilon) \langle \beta^{n-1}\rangle\cdot N_\epsilon.
\end{align*}

On the other hand, we claim that Theorem \ref{thm mov-big} implies
$$
n (N\cdot \langle\alpha^{n-1}\rangle)(\alpha\cdot \langle \beta^{n-1}\rangle) \geq \langle\alpha^{n}\rangle (N\cdot \langle\beta^{n-1}\rangle)
$$
for any nef $(1,1)$-class $N$. First note that both sides of the above inequality are of the same degree of each cohomology class. After scaling, we can assume
$$
\alpha\cdot \langle \beta^{n-1}\rangle=N\cdot \langle\beta^{n-1}\rangle.
$$
Then we need to prove $n N\cdot \langle\alpha^{n-1}\rangle \geq \langle\alpha^{n}\rangle$. Otherwise, we have $n N\cdot \langle\alpha^{n-1}\rangle < \langle\alpha^{n}\rangle$. And Theorem \ref{thm mov-big} implies that there must exist a K\"ahler current in the class $\alpha-N$. Then we must have $$\langle \beta^{n-1}\rangle\cdot (\alpha-N)>0,$$
 which contradicts with our scaling equality
 $\langle \beta^{n-1}\rangle\cdot (\alpha-N)=0$.

Let $N=N_\epsilon$, we get
\begin{align*}
(1+\epsilon)n(N_\epsilon \cdot \langle \beta^{n-1}\rangle) (\alpha\cdot \langle \beta^{n-1}\rangle) &\geq n(N_\epsilon \cdot \langle \alpha^{n-1}\rangle)  (\alpha\cdot \langle \beta^{n-1}\rangle)\\
&\geq \langle\alpha^{n}\rangle (N_\epsilon \cdot \langle\beta^{n-1}\rangle).
\end{align*}
This implies
$$
(1+\epsilon)n \alpha\cdot \langle \beta^{n-1}\rangle \geq \langle\alpha^{n}\rangle.
$$
Since $\epsilon>0$ is arbitrary, this contradicts with our assumption $\langle\alpha^n\rangle-n \alpha\cdot \langle \beta^{n-1}\rangle>0$. Thus there must exist a strictly positive $(n-1,n-1)$-current in the class
$\langle \alpha^{n-1}\rangle-\langle \beta^{n-1}\rangle$.
\end{proof}

\begin{remark}
Let $X$ be a smooth projective variety of dimension $n$ and let $\Mov_1(X)$ be the closure of the cone generated by movable curve classes. In \cite{lehmann2015zariski}, we show that any interior point of $\Mov_1(X)$ is the form $\langle L^{n-1}\rangle$ for a unique big and movable divisor class. And under Demailly's conjecture on transcendental Morse inequality, this also extends to transcendental movable $(n-1, n-1)$-classes over compact K\"ahler manifold. In particular, this extends to compact hyperk\"ahler manifolds.
\end{remark}

\bibliography{reference}
\bibliographystyle{alpha}

\noindent\\

\noindent
\textsc{Fudan University $\&$ Universit\'{e} Grenoble Alpes}\\
\verb"Email: jian.xiao@ujf-grenoble.fr"

\end{document}